\documentclass[reqno,a4paper,draft]{amsart}

\usepackage{enumitem}
\setenumerate{label=\textnormal{(\arabic*)}}

\usepackage{amsmath,amssymb,dsfont,verbatim,bm,array,mathtools}
\usepackage[latin1]{inputenc}
\usepackage{booktabs}
\usepackage[raggedright]{titlesec}
\usepackage{mathtools}

\titleformat{\chapter}[display]
{\normalfont\huge\bfseries}{\chaptertitlename\\thechapter}{20pt}{\Huge}
\titleformat{\section}
{\normalfont\Large\bfseries\center}{\thesection}{1em}{}
\titleformat{\subsection}
{\normalfont\large\bfseries}{\thesubsection}{1em}{}
\titleformat{\subsubsection}[runin]
{\normalfont\normalsize\bfseries}{\thesubsubsection}{1em}{}
\titleformat{\paragraph}[runin]
{\normalfont\normalsize\bfseries}{\theparagraph}{1em}{}
\titleformat{\subparagraph}[runin]
{\normalfont\normalsize\bfseries}{\thesubparagraph}{1em}{}
\titlespacing*{\chapter} {0pt}{50pt}{40pt}
\titlespacing*{\section} {0pt}{3.5ex plus 1ex minus .2ex}{2.3ex plus .2ex}
\titlespacing*{\subsection} {0pt}{3.25ex plus 1ex minus .2ex}{1.5ex plus .2ex}
\titlespacing*{\subsubsection}{0pt}{3.25ex plus 1ex minus .2ex}{1.5ex plus .2ex}
\titlespacing*{\paragraph} {0pt}{3.25ex plus 1ex minus .2ex}{1em}
\titlespacing*{\subparagraph} {\parindent}{3.25ex plus 1ex minus .2ex}{1em}

\input xypic
\xyoption{all}


\newtheorem{theorem}{Theorem}[section]
\newtheorem{lemma}[theorem]{Lemma}
\newtheorem{proposition}[theorem]{Proposition}
\newtheorem{corollary}[theorem]{Corollary}

\theoremstyle{definition}

\theoremstyle{remark}
\newtheorem{remark}[theorem]{Remark}

\DeclareMathOperator{\Jac}{Jac}

\begin{document}
\title{A special case of the two-dimensional Jacobian Conjecture}

\author{Vered Moskowicz}
\address{Department of Mathematics, Bar-Ilan University, Ramat-Gan 52900, Israel.}
\email{vered.moskowicz@gmail.com}

\begin{abstract}
Let
$f: \mathbb{C}[x,y] \to \mathbb{C}[x,y]$
be a $\mathbb{C}$-algebra endomorphism
having an invertible Jacobian.

We show that for such $f$,
if, in addition,
the group of invertible elements of
$\mathbb{C}[f(x),f(y),x][1/v] \subset \mathbb{C}(x,y)$
is contained in 
$\mathbb{C}(f(x),f(y))-0$,
then $f$ is an automorphism.
Here $v \in \mathbb{C}[f(x),f(y)]-0$
is such that
$y = u/v$, with
$u \in \mathbb{C}[f(x),f(y),x]-0$.

Keller's theorem (in dimension two)
follows immediately,
since Keller's condition
$\mathbb{C}(f(x),f(y))=\mathbb{C}(x,y)$
implies that the group of invertible elements of
$\mathbb{C}[f(x),f(y),x][1/v]$
is contained in 
$\mathbb{C}(x,y)-0 = \mathbb{C}(f(x),f(y))-0$.
\end{abstract}

\maketitle

\section{Introduction}
Throughout this note,
$f: \mathbb{C}[x,y] \to \mathbb{C}[x,y]$
is a $\mathbb{C}$-algebra endomorphism
that satisfies
$\Jac(p,q) \in \mathbb{C}^*$,
where
$p:=f(x)$ and $q:=f(y)$.

Formanek's field of fractions theorem
~\cite[Theorem 2]{formanek field of fractions thm}
in dimension two says that
$\mathbb{C}(p,q,x)=\mathbb{C}(x,y)$.
{}From this it is not difficult to obtain that 
$y = u/v$, for some
$u \in \mathbb{C}[p,q,x]-0$
and
$v \in \mathbb{C}[p,q]-0$.

We show in Theorem \ref{my keller} that
for such $f$,
if, in addition,
the group of invertible elements of
$\mathbb{C}[p,q,x][1/v]$
is contained in 
$\mathbb{C}(p,q)-0$,
then $f$ is an automorphism.

Our proof of Theorem \ref{my keller} is almost identical to the proof of
Formanek's automorphism theorem ~\cite[Theorem 1]{formanek automorphism thm}; 
we did only some slight changes in his proof, 
and also used Formanek's field of fractions theorem 
and Wang's intersection theorem ~\cite[Theorem 41 (i)]{wang}.

Keller's theorem in dimension two
follows immediately from our theorem:
Assume that
$\mathbb{C}(p,q)=\mathbb{C}(x,y)$.
Then our condition of Theorem \ref{my keller} is satisfied,
because the group of invertible elements of
$\mathbb{C}[p,q,x][1/v] \subset \mathbb{C}(x,y)$
is contained in 
$\mathbb{C}(x,y)-0 = \mathbb{C}(p,q)-0$.


\section{Preliminaries}

Our Theorem \ref{my keller}
deals with the two-dimensional case only.
However, the results we rely on are valid in any dimension $n$,
so we add the following notation:
$F: \mathbb{C}[x_1,\ldots,x_n] \to \mathbb{C}[x_1,\ldots,x_n]$
is a
$\mathbb{C}$-algebra endomorphism 
that satisfies
$\Jac(F_1,\ldots,F_n) \in \mathbb{C}^*$,
where
$F_1:=F(x_1),\ldots,F_n:=F(x_n)$.
When $n=2$ we will keep the above notation, namely,
$x_1=x, x_2=y, F_1=p, F_2=q$.

\begin{theorem}[Formanek's automorphism theorem] 
Suppose that there is a polynomial 
$W$ in $\mathbb{C}[x_1,\ldots,x_n]$ 
such that
$\mathbb{C}[F_1,\ldots,F_n,W]=
\mathbb{C}[x_1,\ldots,x_n]$.
Then 
$\mathbb{C}[F_1,\ldots,F_n]=
\mathbb{C}[x_1,\ldots,x_n]$,
namely, $F$ is an automorphism.
\end{theorem}

\begin{proof}
See ~\cite[Theorem 1]{formanek automorphism thm}
and
~\cite[page 13, Exercise 9]{essen book}.
\end{proof}

\begin{itemize}
\item If there exists
$w \in \mathbb{C}[x,y]$ 
such that
$\mathbb{C}[p,q,w]=\mathbb{C}[x,y]$,
then 
$\mathbb{C}[p,q]=\mathbb{C}[x,y]$,
namely, $f$ is an automorphism.
\end{itemize}

\begin{theorem}[Formanek's field of fractions theorem]
$$
\mathbb{C}(F_1,\ldots,F_n,x_1,\ldots,x_{n-1})=
\mathbb{C}(x_1,\ldots,x_n).
$$
\end{theorem}

\begin{proof}
See ~\cite[Theorem 2]{formanek field of fractions thm}.
\end{proof}

\begin{itemize}
\item 
$\mathbb{C}(p,q,x)=\mathbb{C}(x,y)$
and 
$\mathbb{C}(p,q,y)=\mathbb{C}(x,y)$.
\end{itemize}

Formanek remarks that when $n=2$,
$\mathbb{C}(p,q,w)=\mathbb{C}(x,y)$,
where $w$ is the image of $x$ under any automorphism of 
$\mathbb{C}[x,y]$;
see
~\cite[page 370, just before Theorem 6]{formanek field of fractions thm}.

The two-dimensional case was already proved by Moh ~\cite[page 151]{moh} and by 
Hamann ~\cite[Lemma 2.1, Proposition 2.1(2)]{hamann}.
Moh and Hamann assumed that $p$ is monic in $y$, 
but this is really not a restriction.

It is easy to see that:
\begin{corollary}\label{cor fractions}
There exist 
$u \in
\mathbb{C}[p,q,x]-0$
and 
$v \in
\mathbb{C}[p,q]-0$
such that
$y = u/v$.
\end{corollary}

\begin{proof}
$y \in 
\mathbb{C}(x,y)=\mathbb{C}(p,q,x)=\mathbb{C}(p,q)(x)$.
Since $x$ is algebraic over $\mathbb{C}(p,q)$,
we have
$\mathbb{C}(p,q)(x)=\mathbb{C}(p,q)[x]$
(see ~\cite[Remark 4.7]{rowen}).
Hence,
$y \in \mathbb{C}(p,q)[x]$.
Therefore, there exist
$a_i, b_i \in \mathbb{C}[p,q]$
($b_i \neq 0$)
such that
$y= \sum (a_i/b_i)x^i$.
Then if we denote
$B =\prod{b_i}$ and 
$B_i$ the product of the $b_j$'s except $b_i$,
we get
$y= (1/B) \sum B_i a_i x^i$.
Just take $v:=B$ and 
$u:=\sum B_i a_i x^i$.
\end{proof}

\begin{theorem}[Wang's intersection theorem]
$\mathbb{C}(F_1,\ldots,F_n) \cap \mathbb{C}[x_1,\ldots,x_n] = 
\mathbb{C}[F_1,\ldots,F_n]$.
\end{theorem}

\begin{proof}
See ~\cite[Theorem 41 (i)]{wang}
and
~\cite[Corollary 1.1.34 (ii)]{essen book}.
\end{proof}

Wang's intersection theorem has a more general version due to Bass
~\cite[Remark after Corollary 1.3, page 74]{bass},
~\cite[Proposition D.1.7]{essen book};
we will not need the more general version here.

\begin{itemize}
\item $\mathbb{C}(p,q) \cap \mathbb{C}[x,y] = \mathbb{C}[p,q]$.
\end{itemize}

The following is immediate:
\begin{corollary}\label{cor intersection}
$\mathbb{C}(p,q) \cap R = \mathbb{C}[p,q]$,
for any 
$\mathbb{C}[p,q] \subseteq R \subseteq \mathbb{C}[x,y]$. 
In particular,
$\mathbb{C}(p,q) \cap \mathbb{C}[p,q,x] = \mathbb{C}[p,q]$.
\end{corollary}

\begin{proof}
$\mathbb{C}(p,q) \cap R \subseteq
\mathbb{C}(p,q) \cap \mathbb{C}[x,y] = \mathbb{C}[p,q]$.
The other inclusion,
$\mathbb{C}(p,q) \cap R \supseteq \mathbb{C}[p,q]$,
is trivial.
\end{proof}

\begin{theorem}[Keller's theorem]
If
$\mathbb{C}(F_1,\ldots,F_n)=\mathbb{C}(x_1,\ldots,x_n)$,
then $F$ is an automorphism.
\end{theorem}

$F$ as in Keller's theorem is called birational
($F$ has an inverse formed of rational functions).

\begin{proof}
See ~\cite{keller}, ~\cite[Corollary 1.1.35]{essen book} 
and ~\cite[Theorem 2.1]{bcw}.
\end{proof}

\begin{itemize}
\item If $\mathbb{C}(p,q)=\mathbb{C}(x,y)$,
then $f$ is an automorphism.
\end{itemize}

\begin{remark}
Notice that the above results are dealing with 
$k[x_1,\ldots,x_n]$, where $k$ is:
\begin{itemize}
\item $\mathbb{C}$: Formanek's field of fractions theorem.
\item a field of characteristic zero: Formanek's automorphism theorem.
\item any field: Keller's theorem.
\item a UFD: Wang's intersection theorem.
\end{itemize}

We have not checked if Formanek's field of fractions theorem
is valid over a more general field than $\mathbb{C}$;
if, for example, it is valid over any algebraic closed field 
of characteristic zero, then our Theorem \ref{my keller}
is valid over any algebraic closed field 
of characteristic zero, not just over $\mathbb{C}$.

Anyway, working over $\mathbb{C}$ is good enough in view of
~\cite[Lemma 1.1.14]{essen book}.
\end{remark}

\section{A new proof of Keller's theorem in dimension two}


Our proof of Theorem \ref{my keller} relies heavily on the proof of Formanek's automorphism theorem;
we did only some slight changes in his proof, 
changes that seem quite natural in view of 
Corollary \ref{cor fractions}:

Although we do not know if 
$\mathbb{C}[p,q,x]=\mathbb{C}[x,y]$
(if so, then $f$ is an automorphism
by Formanek's automorphism theorem),
we do know that $\mathbb{C}(p,q,x)=\mathbb{C}(x,y)$
(by Formanek's field of fractions theorem),
so by Corollary \ref{cor fractions},
$y=u/v$ for some 
$u \in \mathbb{C}[p,q,x]-0$
and 
$v \in \mathbb{C}[p,q]-0$.
Therefore, it seems natural to consider
$\beta: \mathbb{C}[U_1,U_2,U_3][1/V] \to \mathbb{C}[p,q,x][1/v]$,
where $V=v(U_1,U_2)$.

This $\beta$ has $x$ and $y$ in its image, 
so most of Formanek's proof can be adjusted here,
except that the group of invertible elements
of $\mathbb{C}[p,q,x][1/v]$
is not as easily described as 
the group of invertible elements
of $\mathbb{C}[x,y]$,
which is obviously $\mathbb{C}^*$.

Only after adding a condition on the group of invertible elements
of $\mathbb{C}[p,q,x][1/v]$,
we are able to show that $f$ is an automorphism.

Now we are ready to bring our theorem;
we recommend the reader to first read the proof of
Formanek's automorphism theorem, and then read our proof,
with $p,q,x$ in our proof instead of $F_1,F_2,F_3$ in his proof.
 
\begin{theorem}[Main Theorem]\label{my keller}
If the group of invertible elements of
$\mathbb{C}[p,q,x][1/v]$
is contained in 
$\mathbb{C}(p,q)-0$,
then $f$ is an automorphism.
\end{theorem}

\begin{proof}
By Corollary \ref{cor fractions},
there exist 
$u \in \mathbb{C}[p,q,x]-0$
and
$v \in \mathbb{C}[p,q]-0$
such that
$y= u/v$.

Let $U_1,U_2,U_3$ be independent variables over $\mathbb{C}$.
Define 
$\alpha: \mathbb{C}[U_1,U_2,U_3] \to \mathbb{C}[p,q,x]$
by
$\alpha(U_1):=p$,
$\alpha(U_2):=q$,
$\alpha(U_3):=x$.
Clearly, $\alpha$ is surjective.

Claim: The kernel of $\alpha$ is a principal prime ideal
of $\mathbb{C}[U_1,U_2,U_3]$.

Proof of claim: $\mathbb{C}(U_1,U_2,U_3)$ has transcendence degree $3$
over $\mathbb{C}$,
and $\mathbb{C}(p,q,x)=\mathbb{C}(x,y)$ has transcendence degree $2$
over $\mathbb{C}$.
{}From ~\cite[Theorem 5.6]{crt},
$\mathbb{C}[U_1,U_2,U_3]$ is of Krull dimension $3$
and 
$\mathbb{C}[p,q,x]$ is of Krull dimension $2$.
Hence, the kernel of $\alpha$ is of height $1$,
and in a Noetherian UFD a height one prime ideal is principal,
see ~\cite[Theorem 15.9]{pete}.

Denote by $H$ a generator of the kernel of $\alpha$:
$H= H_r U_3^r+\ldots+H_1U_3+H_0$,
where $H_j \in \mathbb{C}[U_1,U_2]$ and $r \geq 1$.
$H$ is a product of the minimal polynomial for $x$ over 
$\mathbb{C}(p,q)$ by some element $H_r$ of $\mathbb{C}[U_1,U_2]$
which clears the denominators of the minimal polynomial 
for $x$ over $\mathbb{C}(p,q)$.
Notice that $r=0$ is impossible, 
since then 
$H= H_0(U_1,U_2)$:
\begin{itemize}
\item
If $H_0(U_1,U_2) \equiv 0$,
then $H(U_1,U_2,U_3) \equiv 0$, 
so the kernel of $\alpha$ is zero,
but then we have
$\mathbb{C}[U_1,U_2,U_3] \cong \mathbb{C}[p,q,x]$,
which is impossible from considerations of Krull dimensions.
\item 
If $H_0(U_1,U_2) \neq 0$,
then 
$0=\alpha(H)=\alpha(H_0(U_1,U_2))=H_0(p,q)$
is a non-trivial algebraic dependence
of $p$ and $q$ over $\mathbb{C}$.
But $p$ and $q$ are algebraically independent over 
$\mathbb{C}$,
because $\Jac(p,q) \neq 0$;
see ~\cite[pages 19-20]{makar} or 
~\cite[Proposition 6A.4]{rowen}.
\end{itemize}

Since we do not know if $y$ is in the image of $\alpha$,
we define the following (surjective)
$\beta: \mathbb{C}[U_1,U_2,U_3][1/V] \to \mathbb{C}[p,q,x][1/v]$
by
$\beta(U_1):=p$,
$\beta(U_2):=q$,
$\beta(U_3):=x$,
$\beta(1/V):= 1/(\beta(V))$,
where
$V:=v(U_1,U_2)$, namely, 
in $v \in \mathbb{C}[p,q]-0$ replace
$p$ by $U_1$ and $q$ by $U_2$
and get $V$.
It is clear that
$\beta(V)= v$,
so                           
$\beta(1/V)= 1/v$.

Notice that $V \in \mathbb{C}[U_1,U_2]$;
the fact that the $U_3$-degree of $V$
is zero will be crucial in what follows.

Now, $y$ is in the image of $\beta$;
indeed, let $U:=u(U_1,U_2,U_3)$, namely, 
in $u \in \mathbb{C}[p,q,x]-0$ replace
$p$ by $U_1$, $q$ by $U_2$ and $x$ by $U_3$,
and get $U$.
Then clearly $\beta(U/V)=u/v=y$.

Take:
$T_1:=U_3$ and $T_2:=U/V$.
Then, $\beta(T_1)=\beta(U_3)=x$,
and 
$\beta(T_2)=\beta(U/V)=u/v=y$.

Each of the following three elements lie in the kernel of $\beta$:
$U_1-p(T_1,T_2)$,
$U_2-q(T_1,T_2)$
and
$U_3-x(T_1,T_2)=U_3-T_1=0$.
Indeed,
$\beta(U_1-p(T_1,T_2))=\beta(U_1)-\beta(p(T_1,T_2))=p-p=0$
and
$\beta(U_2-q(T_1,T_2))=\beta(U_2)-\beta(q(T_1,T_2))=q-q=0$.

Claim: The kernel of $\beta$ is a principal prime ideal of 
$\mathbb{C}[U_1,U_2,U_3][1/V]$, generated by exactly the same
$H \in \mathbb{C}[U_1,U_2,U_3]$ that generates the kernel of 
$\alpha$.

Proof of claim: Assume that $R/V^j$ is in the kernel of $\beta$,
where $R \in \mathbb{C}[U_1,U_2,U_3]$.
We have
$0=\beta(R/V^j)=\beta(R)/\beta(V)^j=\beta(R)/v^j$,
hence $0=\beta(R)$.
Since $\beta$ restricted to $\mathbb{C}[U_1,U_2,U_3]$
is $\alpha$, we get that $R$ belongs to the kernel of 
$\alpha$, hence $R = \tilde{R} H$,
for some $\tilde{R} \in \mathbb{C}[U_1,U_2,U_3]$.
So, 
$R/V^j= \tilde{R} H/ V^j= (\tilde{R}/ V^j) H$,
as claimed.

Therefore, there exist $R_1,R_2 \in \mathbb{C}[U_1,U_2,U_3]$
($R_3=0$) and $n,m \geq 0$  
such that
$U_1-p(T_1,T_2) = (R_1/V^n) H$
and
$U_2-q(T_1,T_2)= (R_2/V^m) H$.
So,
$U_1= p(T_1,T_2) + (R_1/V^n) H$
and
$U_2= q(T_1,T_2) + (R_2/V^m) H$
(and
$U_3= T_1$).

Differentiating these three equations with respect to $U_1,U_2,U_3$
and using the Chain Rule, we get similar matrices to those in Formanek's proof;
the difference is that instead of $R_1, R_2, R_3$ of Formanek's proof,
we have here $R_1/V^n, R_2/V^m, 0$.

Applying $\beta$ gives a matrix equation over 
$\mathbb{C}[p,q,x][1/v]$, similar to the matrix equation 
$(2)$ of Formanek's proof.

Cramer's Rule shows that 
$\beta(\partial H/\partial U_3) = \lambda/ d$,
where $\lambda = \Jac(p,q) \in \mathbb{C}^*$ 
and 
$d \in \mathbb{C}[p,q,x][1/v]-0$ 
is the determinant of the matrix on the left.

$d$ belongs to the group of invertible elements
of $\mathbb{C}[p,q,x][1/v]$,
hence, by our assumption, 
$d$ belongs to $\mathbb{C}(p,q)-0$.

On the one hand, 
$d \in \mathbb{C}[p,q,x][1/v]-0$, 
hence $d= \tilde{d}/v^l$
for some 
$\tilde{d} \in \mathbb{C}[p,q,x]-0$
and $l \geq 0$.
On the other hand,
$d \in \mathbb{C}(p,q)-0$,
hence $d=a/b$
for some $a,b \in \mathbb{C}[p,q]-0$.
Combining the two we get,
$\tilde{d}/v^l = a/b$,
so
$\mathbb{C}[p,q,x]-0 \ni \tilde{d}= 
v^l(a/b) \in \mathbb{C}(p,q)-0$.
{}From Corollary \ref{cor intersection}
we get that
$\tilde{d} \in \mathbb{C}[p,q]-0$.

(Remark: Actually, one can use Wang's intersection theorem directly,
without Corollary \ref{cor intersection},
and still get 
$\tilde{d} \in \mathbb{C}[p,q]-0$,
as long as one observes that
$\mathbb{C}[p,q,x][1/v]= \mathbb{C}[x,y][1/v]$.
Indeed,
$d \in \mathbb{C}[p,q,x][1/v] = \mathbb{C}[x,y][1/v]$, 
hence $d= \tilde{d}/v^l$
for some $\tilde{d} \in \mathbb{C}[x,y]-0$
and $l \geq 0$, etc.).

So $d= \tilde{d}/v^l$,
with $\tilde{d} \in \mathbb{C}[p,q]-0$.
Let 
$D=d(U_1,U_2)=\tilde{d}(U_1,U_2)/v^l(U_1,U_2)=
\tilde{d}(U_1,U_2)/V^l$.
Clearly,
$\beta(D) = d$.

For convenience, multiply the above equation
$\beta(\partial H/\partial U_3) = \lambda/ d$
by $d$ and get
$d \beta(\partial H/\partial U_3) = \lambda$.
Then
$\beta(D) \beta(\partial H/\partial U_3) = \lambda$,
so
$\beta(D \partial H/\partial U_3) = \beta(\lambda)$.
Therefore,
$D \partial H/\partial U_3 - \lambda$
is in the kernel of $\beta$.

We have seen that the kernel of $\beta$ is a principal ideal of 
$\mathbb{C}[U_1,U_2,U_3][1/V]$, generated by 
$H \in \mathbb{C}[U_1,U_2,U_3]$,
hence there exist 
$S \in \mathbb{C}[U_1,U_2,U_3]$ 
and $t \geq 0$
such that
$D \partial H/\partial U_3 - \lambda = (S/V^t)H$.
Replace $D$ by $\tilde{d}(U_1,U_2)/V^l$
and get,
$(\tilde{d}(U_1,U_2)/V^l) \partial H/\partial U_3 - \lambda = (SH)/V^t$.
Multiply both sides by 
$V^{l+t}$ and get,
$V^t \tilde{d}(U_1,U_2) \partial H/\partial U_3 - \lambda V^{l+t} = V^l(SH)$.

Now, as promised above, we use the fact that the 
$U_3$-degree of $V$ is zero:
The $U_3$-degree of the right side is at least $r$
(= that of $H$, which is exactly $r$, plus that of $S$, which is $\geq 0$),
while the $U_3$-degree of the left side is exactly $r-1$
(= that of $\partial H/\partial U_3$).
It follows that $S=0$ and $r-1=0$,
so $r=1$
and 
$H= H_1(U_1,U_2) U_3 + H_0(U_1,U_2)$.
Apply $\beta$ and get
$0 = H_1(p,q) x + H_0(p,q)$,
so
$x = -H_0(p,q)/H_1(p,q) \in \mathbb{C}(p,q)$.
By Wang's intersection theorem, 
$x \in \mathbb{C}[p,q]$.
Then obviously,
$\mathbb{C}[p,q][y]= \mathbb{C}[x,y]$.
Finally, Formanek's automorphism theorem
implies that 
$\mathbb{C}[p,q]= \mathbb{C}[x,y]$,
namely $f$ is an automorphism.
\end{proof}


All the arguments and known results we use do not depend on 
Keller's theorem, hence we have a new proof of Keller's theorem in dimension two:
\begin{theorem}[Keller's theorem]
Let
$f:\mathbb{C}[x,y] \to \mathbb{C}[x,y]$ 
be a $\mathbb{C}$-algebra endomorphism  
that satisfies
$\Jac(p,q) \in \mathbb{C}^*$.
If 
$\mathbb{C}(p,q)=\mathbb{C}(x,y)$,
then $f$ is an automorphism.
\end{theorem}

\begin{proof}
The group of invertible elements of
$\mathbb{C}[p,q,x][1/v] \subset \mathbb{C}(x,y)$
is contained in 
$\mathbb{C}(x,y)-0 = \mathbb{C}(p,q)-0$.
Now apply Theorem \ref{my keller}.
\end{proof}

Notice that the converse of Theorem \ref{my keller} is trivially true:
If $f$ is an automorphism, 
then 
$\mathbb{C}[p,q] = \mathbb{C}[x,y]$,
so
$\mathbb{C}(p,q) = \mathbb{C}(x,y)$,
hence the group of invertible elements of
$\mathbb{C}[p,q,x][1/v] \subset \mathbb{C}(x,y)$
is contained in 
$\mathbb{C}(x,y)-0 = \mathbb{C}(p,q)-0$.

Another argument: 
If $f$ is an automorphism, 
then we can take
$u=y$ and $v=1$.
Then
$\mathbb{C}[p,q,x][1/v]= \mathbb{C}[x,y]$,
and its group of invertible elements 
is $\mathbb{C}^*$,
which is contained in $\mathbb{C}(p,q)-0$.
 
Therefore, the condition in Keller's theorem
is equivalent to our condition, not just implies our condition:
\begin{proposition}
TFAE:
\begin{itemize}
\item [(i)] $f$ is an automorphism, i.e. $\mathbb{C}[p,q]=\mathbb{C}[x,y]$.
\item [(ii)] $f$ is birational, i.e. $\mathbb{C}(p,q)=\mathbb{C}(x,y)$.
\item [(iii)] The group of invertible elements of
$\mathbb{C}[p,q,x][1/v]$
is contained in 
$\mathbb{C}(p,q)-0$.
\end{itemize}
\end{proposition}
We do not know how to show directly that $(iii)$ implies $(ii)$.

\section{Further discussion}

We wish to bring some related ideas.

\textbf{First idea:} 
We have already mentioned in the Preliminaries 
that Formanek remarks that 
$\mathbb{C}(p,q,w)=\mathbb{C}(x,y)$,
where $w$ is the image of $x$ under any automorphism of 
$\mathbb{C}[x,y]$.
Therefore, we can obtain similar theorems
to Theorem \ref{my keller} with $x$
replaced by any image of $x$ under an 
automorphism of $\mathbb{C}[x,y]$.

More elaborately, take any automorphism 
$g: \mathbb{C}[x,y] \to \mathbb{C}[x,y]$
and denote 
$g_1:=g(x)$ and $g_2:=g(y)$.

We have
$\mathbb{C}(p,q,g_1)=\mathbb{C}(x,y)=
\mathbb{C}(g_1,g_2)$;
the first equality follows from Formanek's remark,
while the second equality trivially follows from 
$\mathbb{C}[x,y]=\mathbb{C}[g_1,g_2]$.
Then,
$g_2 \in \mathbb{C}(p,q)(g_1) = \mathbb{C}(p,q)[g_1]$,
because $g_1$ is algebraic over $\mathbb{C}(p,q)$.
It is easy to obtain
$g_2 = u_g/v_g$,
where
$u_g \in \mathbb{C}[p,q,g_1]-0$
and
$v_g \in \mathbb{C}[p,q]-0$.

\begin{theorem}
If the group of invertible elements of
$\mathbb{C}[p,q,g_1][1/v_g]$
is contained in 
$\mathbb{C}(p,q)-0$,
then $f$ is an automorphism.
\end{theorem}

\begin{proof}
In the proof of Theorem \ref{my keller}
replace $x$ and $y$ by $g_1$ and $g_2$, 
do the appropriate adjustments,
and get a proof for the new theorem.
Notice that now, instead of considering 
$p$ and $q$ as functions of $x$ and $y$,
one has to consider 
$p$ and $q$ as functions of $g_1$ and $g_2$.
\end{proof}

\textbf{Second idea:} 
For $v$ as in Corollary \ref{cor fractions}
write
$v= v_1 \cdots v_m$,
where $v_1,\ldots,v_m \in \mathbb{C}[p,q]$ 
are irreducible elements of 
$\mathbb{C}[p,q]$.
There are two options, either 
one (or more) of the $v_j$'s
becomes reducible in $\mathbb{C}[x,y]$
or all the $v_j$'s remain irreducible
in $\mathbb{C}[x,y]$.

If one (or more) of the $v_j$'s
becomes reducible in $\mathbb{C}[x,y]$,
then it is possible to show that
our condition of Theorem \ref{my keller}
is not satisfied,
and hence $f$ is not an automorphism:
Assume that
$v_1=w_1 \cdots w_l$, where 
$w_1, \ldots, w_l \in \mathbb{C}[x,y]$
are irreducible in 
$\mathbb{C}[x,y]$,
$l > 1$.
It is not difficult to see 
(use Wang's intersection theorem)
that at least two factors are in
$\mathbb{C}[x,y] - \mathbb{C}[p,q]$,
w.l.o.g $w_1$ and $w_2$.
We claim that $w_1$ is invertible in 
$\mathbb{C}[p,q,x][1/v] = \mathbb{C}[x,y][1/v]$.
Indeed,
$1=v/v= v_1 \cdots v_m / v =
w_1 w_2 \cdots w_l v_2 \cdots v_m / v =
w_1 (w_2 \cdots w_l v_2 \cdots v_m / v)$.

Clearly, 
$w_1 \notin \mathbb{C}(p,q)$,
because otherwise,
$w_1 \in \mathbb{C}(p,q) \cap \mathbb{C}[x,y] = \mathbb{C}[p,q]$,
but 
$w_1 \in \mathbb{C}[x,y] - \mathbb{C}[p,q]$.

Actually, if one (or more) of the $v_j$'s
becomes reducible in $\mathbb{C}[x,y]$,
then it is immediate that 
$f$ is not an automorphism,
since an automorphism satisfies
$\mathbb{C}[p,q]=\mathbb{C}[x,y]$,
so trivially every irreducible element of 
$\mathbb{C}[p,q]$ 
is an irreducible element of 
$\mathbb{C}[x,y]$.

Next, if all the $v_j$'s remain irreducible
in $\mathbb{C}[x,y]$,
then our condition of Theorem \ref{my keller}
is satisfied:

\begin{theorem}[A special case of the main theorem]\label{my keller special case}
If $v_1,\ldots,v_m$ remain irreducible in 
$\mathbb{C}[x,y]$,
then $f$ is an automorphism.
\end{theorem}

Of course, since $\mathbb{C}[p,q]$ ($\mathbb{C}[x,y]$) 
is a UFD,
every irreducible element of 
$\mathbb{C}[p,q]$ ($\mathbb{C}[x,y]$) is prime.

\begin{proof}
By assumption,
$v_1,\ldots,v_m \in \mathbb{C}[p,q]$ 
are irreducible elements of
$\mathbb{C}[x,y]$,
hence,
$v_1,\ldots,v_m$ 
are prime elements of
$\mathbb{C}[x,y]$.

Claim: The condition of Theorem \ref{my keller}
is satisfied.

Proof of claim:
Let 
$a \in \mathbb{C}[p,q,x][1/v]= \mathbb{C}[x,y][1/v]$
be an invertible element,
so there exists 
$b \in \mathbb{C}[p,q,x][1/v]= \mathbb{C}[x,y][1/v]$
such that
$ab = 1$.
We can write
$a= r/v^k$
and
$b= s/v^l$,
for some
$r,s \in \mathbb{C}[x,y]-0$
and
$k,l \geq 0$.
Then $ab=1$ 
becomes
$rs = v^{k+l} = (v_1 \cdots v_m)^{k+l}$.
Since 
$v_1,\ldots,v_m$ are prime elements of
$\mathbb{C}[x,y]$,
we obtain that
$r= v_1^{\alpha_1} \cdots v_m^{\alpha_m}$
and
$s= v_1^{\beta_1} \cdots v_m^{\beta_m}$,
where 
$\alpha_j+\beta_j = k+l$,
$1 \leq j \leq m$.
Therefore,
$r,s \in \mathbb{C}[p,q]-0$,
so
$a= r/v^k \in \mathbb{C}(p,q)-0$,
and we are done.
\end{proof}

Notice that in Theorem \ref{my keller special case}
we demand that each of the irreducible factors 
$v_1,\ldots,v_m \in \mathbb{C}[p,q]$
of $v$
remain irreducible in $\mathbb{C}[x,y]$,
but we do not demand that other irreducible elements of 
$\mathbb{C}[p,q]$
remain irreducible in 
$\mathbb{C}[x,y]$.

If one demands that every irreducible element of 
$\mathbb{C}[p,q]$
remains irreducible in 
$\mathbb{C}[x,y]$,
then, without relying on Theorem \ref{my keller},
one can get that $f$ is an automorphism,
thanks to the result ~\cite[Lemma 3.2]{J}
of Jedrzejewicz and Zieli\'{n}ski. 

Their result says the following:
Let $A$ be a UFD. 
Let $R$ be a subring of $A$ such that
$R^* = A^*$. 
The following conditions are equivalent:
\begin{itemize}
\item [(i)] Every irreducible element of $R$ 
remains irreducible in $A$.
\item [(ii)] $R$ is factorially closed in $A$.
\end{itemize}

(Recall that a sub-ring $R$ of a ring $A$ 
is called factorially closed in $A$
if whenever $a_1,a_2 \in A$
satisfy $a_1 a_2 \in R-0$,
then $a_1,a_2 \in R$).

In ~\cite[Lemma 3.2]{J} take
$A=\mathbb{C}[x,y], R=\mathbb{C}[p,q]$;
since we now assume that every irreducible element of 
$\mathbb{C}[p,q]$
remains irreducible in 
$\mathbb{C}[x,y]$, 
we obtain that
$\mathbb{C}[p,q]$ is factorially closed in
$\mathbb{C}[x,y]$,
and we are done by the following easy lemma:

\begin{lemma}\label{lemma}
If $\mathbb{C}[p,q]$ is factorially closed in
$\mathbb{C}[x,y]$,
then $f$ is an automorphism.
\end{lemma}

\begin{proof}
Let $H$ be as in the proof of Theorem \ref{my keller},
and denote 
$h_j := H_j(p,q)$, $0 \leq j \leq r$.
Obviously, $h_0 \neq 0$ by the minimality of $r$.

We have
$x(h_rx^{r-1}+h_{r-1}x^{r-2}+\ldots+h_1)=
-h_0 \in \mathbb{C}[p,q]-0$.
By assumption
$\mathbb{C}[p,q]$ is factorially closed in
$\mathbb{C}[x,y]$,
hence
$x \in \mathbb{C}[p,q]$
(and 
$h_rx^{r-1}+h_{r-1}x^{r-2}+\ldots+h_1 \in \mathbb{C}[p,q]$).

Then
$\mathbb{C}[p,q,y]=\mathbb{C}[x,y]$,
and $f$ is an automorphism
by Formanek's automorphism theorem.
\end{proof}

Notice that in the proof of Lemma \ref{lemma},
$h_rx^{r-1}+h_{r-1}x^{r-2}+\ldots+h_1 \in \mathbb{C}[p,q]$
also yields that $f$ is an automorphism,
because by the minimality of $r$,
we must have $r=1$,
so $h_1 x + h_0 =0$.
Then $x = -h_0 / h_1 \in \mathbb{C}(p,q)$,
and by Wang's intersection theorem,
$x \in \mathbb{C}[p,q]$, etc.

\textbf{Third idea:} Notations as in the second idea,
another special case is when all the $v_j$'s are primes
in 
$\mathbb{C}[p,q,x]$;
this special case is dealt with in
~\cite{vered}:
It is shown in ~\cite[Theorem 2.2]{vered}
that if all the $v_j$'s are primes
in $\mathbb{C}[p,q,x]$, then 
$\mathbb{C}[p,q,x]$ is a UFD,
and it is shown in ~\cite[Theorem 2.1]{vered}
that if  
$\mathbb{C}[p,q,x]$ is a UFD,
then $f$ is an automorphism.

It is not yet clear to us what happens in 
the more general case when
all the $v_j$'s are irreducibles
in $\mathbb{C}[p,q,x]$.
It may happen that some (or all)
of the $v_j$'s are not primes in 
$\mathbb{C}[p,q,x]$,
since we just know that 
$\mathbb{C}[p,q,x]$ is an integral domain
(if we knew it is a UFD, then $f$ is an automorphism
by ~\cite[Theorem 2.1]{vered}).
 
\textbf{Fourth idea:} 

We do not know if a similar result to Theorem \ref{my keller} 
holds in higher dimensions.
Even if the answer is positive, the proof should be somewhat different
from the proof of the two-dimensional case.
For example,
already in the three-dimensional case some problems may arise
when trying to generalize the proof of the two-dimensional case:

Let
$f: \mathbb{C}[x,y,z] \to \mathbb{C}[x,y,z]$
be a $\mathbb{C}$-algebra endomorphism 
having an invertible Jacobian.
Denote 
$p:= f(x), q:=f(y), r:=f(z)$.

It is not difficult to generalize Corollary \ref{cor fractions}:
\begin{corollary}
There exist 
$u \in
\mathbb{C}[p,q,r,x,y]-0$
and 
$v \in
\mathbb{C}[p,q,r]-0$
such that
$z = u/v$.
\end{corollary}

\begin{proof}
By Formanek's field of fractions theorem,
$\mathbb{C}(p,q,r,x,y)= \mathbb{C}(x,y,z)$.

Since $x$ and $y$ are algebraic over
$\mathbb{C}(p,q,r)$, 
a generalization of ~\cite[Remark 4.7]{rowen}
implies that
$\mathbb{C}(p,q,r)[x,y]= \mathbb{C}(p,q,r)(x,y)$.
Then, 
$\mathbb{C}(p,q,r)[x,y]= \mathbb{C}(x,y,z) \ni z$.
{}From this it is not difficult to obtain that
$z =u/v$,
where
$u \in
\mathbb{C}[p,q,r,x,y]-0$
and 
$v \in
\mathbb{C}[p,q,r]-0$.
\end{proof}

Define: 
$\alpha: 
\mathbb{C}[U_1,U_2,U_3,U_4,U_5] \to \mathbb{C}[p,q,r,x,y]$
by
$\alpha(U_1):=p$,
$\alpha(U_2):=q$,
$\alpha(U_3):=r$,
$\alpha(U_4):=x$,
$\alpha(U_5):=y$.
Clearly, $\alpha$ is surjective.

We can define 
$\beta: \mathbb{C}[U_1,U_2,U_3,U_4,U_5][1/V] 
\to \mathbb{C}[p,q,r,x,y][1/v]$.
It is clear that
$z \in \mathbb{C}[p,q,r,x,y][1/v]$.

The kernel of $\alpha$ is a height two prime ideal;
indeed, 
$\mathbb{C}[U_1,U_2,U_3,U_4,U_5]$ is of Krull dimension $5$
and $\mathbb{C}[p,q,r,x,y]$ is of Krull dimension $3$,
hence the kernel of $\alpha$ is of height two.

{}From Krull's principal ideal theorem
~\cite[Theorem 8.42]{pete},
the kernel of $\alpha$ is generated by at least two elements.

Assume for the moment that the kernel of $\alpha$ 
is generated by exactly two elements,
hence the kernel of $\beta$ 
is generated by the same two elements.
The matrix equation $(2)$ in Formanek's proof
will involve the product of a 
$5 \times 5$ matrix with a $5 \times 5$ matrix,
but Cramer's Rule seems not to help here.
We do not know yet if it is possible to overcome this problem.

\bibliographystyle{plain}

\end{document}